\documentclass[a4paper,12pt,reqno]{amsart} 
\usepackage{amssymb,amsthm,amsmath}
\usepackage{mathtools}
\usepackage{amsfonts}
\usepackage{amscd}
\usepackage{amssymb}

\usepackage{enumerate}
\usepackage{bbm}
\usepackage{graphicx}
\allowdisplaybreaks
\usepackage{mathabx}
\usepackage{color}
\usepackage{amsbsy}
\usepackage{graphicx}
\usepackage{amsthm}
\usepackage{amsmath}
\usepackage{amsxtra}
\usepackage{mathrsfs}
\usepackage{bbm}
\usepackage{dsfont}
\usepackage{enumitem}
\usepackage{comment}

\usepackage{enumerate}
\usepackage{xcolor}
\usepackage{float} 

\usepackage{url}
\usepackage{soul}
\usepackage[hidelinks]{hyperref} 

\usepackage{ifthen}

\hypersetup{
    colorlinks=true,
    linkcolor=red,      
    citecolor=red,       
    urlcolor=cyan        
}

\newtheorem{theorem}{Theorem}[section]

\newtheorem{thm}{Theorem}[subsection]
\newtheorem{defn}[thm]{Definition}
\newtheorem{remark}[theorem]{Remark}

\newtheorem{corollary}[theorem]{Corollary}

\newtheorem{lemma}[theorem]{Lemma}

\numberwithin{equation}{section}
\definecolor{red}{rgb}{1.0, 0.0, 0.0}
\newcommand{\Bea}{\begin{eqnarray*}}
	\newcommand{\Eea}{\end{eqnarray*}}
\newcommand{\Be} {\begin{equation*}}
	\newcommand{\Ee} {\end{equation*}}
\newcommand{\be} {\begin{equation}}
	\newcommand{\ee} {\end{equation}}
\newcommand{\bea} {\begin{eqnarray}}
	\newcommand{\eea} {\end{eqnarray}}

\newcommand{\la}{\lambda}

\usepackage{color}

\title[ Boundedness of multipliers for the Strichartz Fourier transform ]{ Boundedness of Fourier Multipliers and Applications to Nonlinear PDEs for the Strichartz Fourier Transform on the Heisenberg Group}
\author[A. Dasgupta]{Aparajita Dasgupta}
\address{
	Aparajita Dasgupta:
	\endgraf
	Department of Mathematics
	\endgraf
	Indian Institute of Technology Delhi, Hauz Khas
	\endgraf
	New Delhi-110016 
	\endgraf
	India
	\endgraf
	{\it E-mail address:} {\rm adasgupta@maths.iitd.ac.in}
}
\author[P. Gulia]{Prerna Gulia}
\address{
	Prerna Gulia:
	\endgraf
	Department of Mathematics
	\endgraf
	Indian Institute of Technology Delhi, Hauz Khas
	\endgraf
	New Delhi-110016 
	\endgraf
	India
	\endgraf
	{\it E-mail address:} {\rm prernagulia64@gmail.com}
}


\vspace{1cm}

\allowdisplaybreaks
\begin{document}
\begin{abstract}
We investigate Fourier multipliers associated with the Strichartz Fourier transform on the Heisenberg group. In particular, we establish H\"ormander-type $L^{p}-L^{q}$ boundedness results for the range $1<p\leq 2\leq q<\infty$. The analysis is based on deriving suitable analogues of the Hausdorff-Young and Paley inequalities for the Strichartz Fourier transform, followed by interpolation arguments to obtain the desired multiplier estimates. As an application, we study the local well-posedness of certain nonlinear partial differential equations. Furthermore, we establish an $L^{p}$-boundedness theorem for Fourier multipliers associated with the Strichartz Fourier transform for the full range $1<p<\infty$.

\end{abstract}
\maketitle

\section{introduction}
The study of bounded translation-invariant operators on \(L^p\) spaces is a central theme in harmonic analysis. In his seminal work, H\"ormander \cite{hormander-estimates_of_translation} initiated the systematic study of such operators on \(\mathbb{R}^n\) via the Euclidean Fourier transform, giving rise to the theory of Fourier multipliers. It is well known that a Fourier multiplier operator can be bounded from \(L^p(\mathbb{R}^n)\) to \(L^q(\mathbb{R}^n)\) only when \(p\leq q\). Two fundamental results in this direction are H\"ormander’s multiplier theorem \cite{hormander-estimates_of_translation}, established for the range \(1<p\leq 2\leq q<\infty\), and Lizorkin’s multiplier theorem \cite{Lizorkin}, which is valid for \(1<p\leq q<\infty\). The latter, however, requires stronger smoothness assumptions on the multiplier symbol. Motivated by this distinction, the present work is primarily devoted to the study of H\"ormander-type multiplier theorems.

We now recall H\"ormander’s multiplier theorem on \(\mathbb{R}^n\). Let \(\sigma:\mathbb{R}^n\to \mathbb{C}\) be a bounded function and consider the Fourier multiplier operator \(T_{\sigma}\), initially defined on the Schwartz space \(S(\mathbb{R}^n)\), by
\begin{equation*}
    \widehat{T_{\sigma}f}(\xi)=\sigma(\xi)\widehat{f}(\xi), \qquad \xi\in \mathbb{R}^n,
\end{equation*}
where \(\widehat{f}\) denotes the Euclidean Fourier transform of \(f\).

H\"ormander proved that if \(1<p\leq 2\leq q<\infty\) and the symbol \(\sigma\) satisfies
\begin{equation*}
    \sup_{\alpha>0}\alpha
    \left(
    \int_{\{\xi:\,|\sigma(\xi)|>\alpha\}} d\xi
    \right)^{1/p-1/q}
    <\infty,
\end{equation*}
then the operator \(T_{\sigma}\) extends to a bounded operator from \(L^p(\mathbb{R}^n)\) to \(L^q(\mathbb{R}^n)\).



The theory of \(L^p-L^q\) multipliers for locally compact groups and compact homogeneous manifolds has been studied in \cite{Rauan-Ruzhansky_multipliers} and \cite{michael-hardy-littlewoord-Lp-Lq-compact-manifolds} respectively. The authors observed that H\"{o}rmander's multiplier theorem can be reformulated in the language of Lorentz spaces and group von Neumann algebras. Rather than studying a Fourier multiplier solely through its symbol, one may regard the multiplier itself as an operator belonging to an appropriate noncommutative function space. The importance of this reformulation lies in the fact that it is not tied to the Euclidean structure of $\mathbb{R}^{n}$. Consequently, it provides a framework for extending H\"{o}rmander-type multiplier theorems to a much broader class of settings, including unimodular locally compact groups and homogeneous spaces. 
Subsequently, Cardona et al.~\cite{cardona-vishvesh-ruzhansky-multipliers} investigated the \(L^p-L^q\) boundedness of pseudo-differential operators on smooth manifolds, while Kumar and Ruzhansky \cite{Vishvesh-Ruzhansky-multipliers_compact_hypergroups} also developed an analogous theory for compact hypergroups. Multiplier operators associated with the anharmonic oscillator were studied by Chatzakou and Kumar \cite{chatzakou-kumar-anharmonic-multipliers}, and later Kumar and Ruzhansky \cite{Vishvesh-ruzhsnaky-k-a_multipliers} established \(L^p-L^q\) boundedness results for \((k,a)\)-Fourier multipliers together with applications to nonlinear partial differential equations, similar to the results in \cite{cardona-vishvesh-ruzhansky-multipliers}. The corresponding problem for the noncommutative torus was investigated in \cite{shaimardan-tulenov-noncomm_torus}. We refer the reader to the references therein for further developments and related results. In the present work, we focus on the setting of the Heisenberg group.


The Heisenberg group \(\mathbb{H}^n\) is one of the most fundamental examples of a non-abelian nilpotent Lie group. Its underlying manifold is \(\mathbb{C}^n\times \mathbb{R}\), equipped with the group law
\[
(z,t)(w,s)=\left(z+w,t+s+\frac{1}{2}\operatorname{Im}(z\cdot \overline{w})\right),
\]
where \((z,t),(w,s)\in \mathbb{C}^n\times \mathbb{R}\).

For a suitable function \(f\) on \(\mathbb{H}^n\), the Fourier transform of \(f\) is an operator-valued function acting on \(L^2(\mathbb{R}^n)\), defined by
\[
\widehat{f}(\lambda)=\int_{\mathbb{C}^n\times \mathbb{R}}f(z,t)\pi_{\lambda}(z,t)\,dz\,dt,
\]
where, for \(\lambda\in \mathbb{R}^{\ast}\), \(\pi_{\lambda}\) denotes the Schr\"odinger representation of \(\mathbb{H}^n\) on \(L^2(\mathbb{R}^n)\). More explicitly, for \(\varphi\in L^2(\mathbb{R}^n)\),
\[
\pi_{\lambda}(z,t)\varphi(\xi)
=e^{i\lambda t}e^{i\lambda\left(x\cdot \xi+\frac12 x\cdot y\right)}
\varphi(\xi+y),
\]
where \(z=x+iy\).

It follows directly from the definition that
\[
\widehat{f}(\lambda)
=\int_{\mathbb{C}^n}f^{\lambda}(z)\pi_{\lambda}(z,0)\,dz,
\]
where \(f^{\lambda}\) denotes the inverse Fourier transform of \(f\) in the central variable, namely,
\[
f^{\lambda}(z)
=\int_{-\infty}^{\infty}f(z,t)e^{i\lambda t}\,dt.
\]

Apart from the \(L^p-L^q\) multiplier theorem, H\"ormander and Mikhlin \cite{hormander-estimates_of_translation,michlin-multipliers} established \(L^p\)-boundedness results for Fourier multipliers, which subsequently motivated extensive developments in several mathematical settings; see for instance \cite{Anker-multipliers,Hejna-multipliers_dunkl_transform,multiplier-hankel,Soltani-multipliers_dunkl}. The study of \(L^p\) Fourier multipliers on the Heisenberg group associated with the classical operator-valued Fourier transform was initiated by Mauceri and Michele \cite{Mauceri-heisenberg-multipliers}, who established the \(L^p\) multiplier theorem in the case \(n=1\). Their approach relied on a detailed analysis of the operator-valued Fourier transform together with the associated difference-differential operators arising from the representation theory of the Heisenberg group.

This work was later extended to higher dimensions by Lin \cite{Lin-multipliers-Heisenberg-group}, who introduced suitable difference-differential operators adapted to the Fourier transform through Fock space representations and established both \(L^p\) multiplier theorems and Hardy space estimates. More recently, Bagchi \cite{bagchi-multipliers} derived explicit formulas for these difference-differential operators and used them to obtain a significantly simpler proof of the multiplier theorem on the Heisenberg group. We refer the reader to \cite{bagchi-multipliers,Muller-Marcinkiewicz_multipliers-I,Muller-Marcinkiewicz_multipliers-II,Thangavelu-multiplier_sublaplacian} for further developments on multipliers associated with the Heisenberg group and the more general class of H-type groups.

In 2023, Thangavelu \cite{Strichartz_Fourier_thangavelu} introduced a new formulation of the Fourier transform on the Heisenberg group, namely the \emph{Strichartz Fourier transform}. This transform is scalar-valued and is constructed using the joint eigenfunctions of the sublaplacian \(\mathcal{L}\) and the vector field \(T=\frac{\partial}{\partial t}\). More precisely, it is built upon the joint spectral decomposition of the commuting operators \(\mathcal{L}\) and \(T\), and therefore provides an alternative spectral realization of harmonic analysis on the Heisenberg group.

The classical Fourier transform on the Heisenberg group, arising from the Schr\"odinger representations, is operator-valued and enjoys fundamental analytical properties such as the inversion formula and the Plancherel theorem. However, its operator-valued nature often introduces substantial technical difficulties in applications. In particular, it becomes less convenient in problems involving the characterization of function spaces, weighted inequalities, uncertainty principles, and other questions where scalar-valued transforms are more natural. The Strichartz Fourier transform overcomes these limitations by providing a scalar-valued analogue while retaining the essential harmonic analytic structure of the Heisenberg group.

Since its introduction, the Strichartz Fourier transform has already led to several developments in harmonic analysis. In collaboration with Pusti and Thangavelu, the authors established weighted norm inequalities for the Strichartz Fourier transform in \cite{Gulia-weighted_norm_inequalities}. Subsequently, in joint work with Dabra \cite{Gulia-strichartz_uncertainty}, several uncertainty principles in the Strichartz setting were obtained. These developments demonstrate that the transform provides a fruitful framework for extending classical harmonic analytic questions to the Heisenberg group.

Another remarkable feature of the Strichartz Fourier transform is its close resemblance to the Helgason Fourier transform on noncompact Riemannian symmetric spaces. Both transforms are scalar-valued and arise from spectral decompositions associated to invariant differential operators in respective spaces. This analogy naturally suggests the possibility of developing a multiplier theory for the Strichartz Fourier transform analogous to that available for the Helgason Fourier transform.

Motivated by this connection, we briefly recall some developments in multiplier theory on Riemannian symmetric spaces. In his seminal work, Anker \cite{Anker-multipliers} established a H\"ormander--Mikhlin type \(L^p\) multiplier theorem for general noncompact Riemannian symmetric spaces, extending the earlier results of Clerc and Stein \cite{Clerc-Stein-multipliers_symmetric_spaces}. This line of research was subsequently developed in \cite{Ionescu-singular_integraks_symmetric_spaces,Wrobel-multipliers_symm_spaces}. More recently, Rana and Ruzhansky \cite{Tapendu-multipliers} established H\"ormander-type \(L^p-L^q\) multiplier theorems for Riemannian symmetric spaces of noncompact type by developing suitable Hausdorff-Young-Paley inequalities and interpolation techniques.

Despite the recent progress in the analysis of the Strichartz Fourier transform, the theory of Fourier multipliers in this framework remains completely undeveloped. To the best of our knowledge, no multiplier theorem has been established for the Strichartz Fourier transform. In particular, analogues of H\"ormander-type \(L^p-L^q\) multiplier theorems, \(L^p\)-multiplier theorems, or multiplier applications to nonlinear equations are absent from the literature. Thus, the present work initiates the systematic study of Fourier multipliers associated with the Strichartz Fourier transform on the Heisenberg group.

The novelty of this work is twofold. First, we develop the multiplier theory in the setting of the Strichartz Fourier transform, thereby extending harmonic analysis in this framework beyond weighted inequalities and uncertainty principles. Second, unlike the classical multiplier theory on the Heisenberg group, which fundamentally relies on the operator-valued Fourier transform and the associated representation-theoretic machinery, the present setting requires an entirely different approach due to the scalar-valued nature of the transform. Consequently, the multiplier techniques developed for the classical Fourier transform on the Heisenberg group, based on representation-theoretic methods and difference-differential operators, are not directly adaptable to the Strichartz framework due to its scalar-valued nature.

To overcome these difficulties, we develop new analytical tools adapted to the Strichartz framework. In particular, we establish suitable analogues of the Hausdorff--Young and Paley inequalities for the Strichartz Fourier transform, which are of independent interest. Combining these inequalities with interpolation methods inspired by multiplier theory on symmetric spaces, we prove H\"ormander-type \(L^p-L^q\) boundedness results for Fourier multipliers. Furthermore, we establish \(L^p\)-boundedness theorems and apply the multiplier results to study local well-posedness of certain nonlinear partial differential equations on the Heisenberg group.

Thus, this work not only provides the first multiplier theory for the Strichartz Fourier transform but also reveals a new interaction between harmonic analysis on the Heisenberg group, multiplier theory on symmetric spaces, and nonlinear partial differential equations.
As mentioned earlier, general multiplier framework developed in \cite{Rauan-Ruzhansky_multipliers} already encompasses the Heisenberg group through the theory of left invariant operators affiliated with the group von Neumann algebra, the present work is formulated in the setting of the scalar valued Strichartz Fourier transform. As an application of \(L^p-L^q\) boundedness theorem, we study local well-posedness of nonlinear PDEs. Similar results have been obtained for smooth manifolds in \cite{cardona-vishvesh-ruzhansky-multipliers}, for \((k,a)\)-Fourier transform in \cite{Vishvesh-ruzhsnaky-k-a_multipliers} and for noncommutative torus in \cite{shaimardan-tulenov-noncomm_torus}.


We conclude the introduction by giving a brief outline of the paper. Section \ref{sec2} provides the necessary background on the Strichartz Fourier transform, mixed norm spaces, Lorentz spaces and interpolation theorems. In Section \ref{sec3}, we establish \(L^p-L^q\) boundedness of Fourier multiplier for the Strichartz Fourier transform and use it to study well-possedness of nonlinear PDEs in Section \ref{sec4}.
Lastly, in Section \ref{sec5}, we obtain \(L^p\) boundedness of the Fourier multiplier for the Strichartz Fourier transform.
\section{preliminaries}\label{sec2}
In this section, we recall some basic results from harmonic analysis on the Heisenberg 
group that play an important role in the study of the Strichartz Fourier transform.

The Heisenberg group is a two-step nilpotent Lie group equipped with the Lebesgue measure \(dz\,dt\), which serves as its Haar measure. The representation theory of \(\mathbb{H}^n\) is well understood, primarily due to the Stone--von Neumann theorem. For a comprehensive account of the representation theory of \(\mathbb{H}^n\), we refer the reader to \cite{folland-harmonic_phase_space, thangaveluheisenberg}.

We denote by \(\mathfrak{h}_n\) the Heisenberg algebra, whose basis is formed by the following left-invariant vector fields
\begin{equation*}
    X_i=\frac{\partial}{\partial x_i}+\frac{1}{2}y_i\frac{\partial}{\partial t},\; Y_i=\frac{\partial}{\partial y_i}-\frac{1}{2}x_i\frac{\partial}{\partial t},\; T= \frac{\partial}{\partial t},\; i=1,2,\dots , n.
\end{equation*}
The sublaplacian on \(\mathbb{H}^n\) is then defined as 
$$
\mathcal{L}:=-\sum_{i=1}^n\left( X_i^2+Y_i^2 \right)
$$
and has explicit form
$$
\mathcal{L} = -\Delta_{\mathbb{C}^n}
- \frac{1}{4} |z|^2 \frac{\partial^2}{\partial t^2}
+ N \frac{\partial}{\partial t},
$$
where $\Delta_{\mathbb{C}^n}$ denotes the Laplacian on $\mathbb{C}^n$ and
$$
N = \sum_{j=1}^n \left( x_j \frac{\partial}{\partial y_j}
- y_j \frac{\partial}{\partial x_j} \right).
$$ 
It is known that \(\mathcal{L}\) commutes with the vector field \(T\), allowing the study of the joint spectral theory of these two operators. In this framework, Laguerre functions arise naturally as they characterize the joint eigenfunctions of $\mathcal{L}$ and the left invariant vector field $T$. We now recall the definition of the Laguerre functions.

\noindent For \(\delta >-1\) and \(k \in \mathbb{N} \cup \{0\}\), let \(L^{\delta}_k\) denote the Laguerre polynomial of type \(\delta\) and degree \(k\), defined as
$$
L^{\delta}_k(t)=\frac{t^{-\delta}e^t}{k!}\left ( \frac{d}{dt} \right )^k \left( e^{-t}t^{k+\delta}\right),\; t\geq 0.
$$
For \(\lambda \in \mathbb{R}^*=\mathbb{R}\setminus \{0\}\), let \(\varphi^{n-1}_{k,\lambda}\) denote the Laguerre function (for more details see \cite{thangaveluheisenberg}) on \(\mathbb{C}^n\) defined as 
$$
\varphi^{n-1}_{k,\lambda}(z)=L^{n-1}_k\left( \frac{1}{2}|\lambda||z|^2\right) e^{-\frac{1}{4}|\lambda||z|^2}.
$$
It is easy to check that 
\begin{equation}\label{twonormofLF}
    \int_{\mathbb{C}^n}|\varphi^{n-1}_{k,\la}(z)|^2\,dz = (2\pi)^n |\la|^{-n} \frac{(k+n-1)!}{k!\,(n-1)!}.
\end{equation}
A detailed study of the eigenfunctions of the sublaplacian can be found in \cite{thangavelu-uncertainty-book}. 
In particular, the functions \(e^{n-1}_{k,\lambda}(z,t):=e^{i \lambda t}\varphi^{n-1}_{k,\lambda}(z)\) arise as joint eigenfunctions of the operators \(\mathcal{L}\) and \(T\). More precisely, they satisfy the following eigenvalue equations
$$
\mathcal{L}(e^{n-1}_{k,\lambda})=(2k+n)\lvert \lambda \rvert \,e^{n-1}_{k,\lambda} \;\;\;\; \text{and} \;\;\;-i\frac{\partial}{\partial t}(e^{n-1}_{k,\lambda})=\la \, e^{n-1}_{k,\lambda} .
$$

We next describe the Strichartz Fourier transform and the associated Heisenberg fan. Let \(\Omega\) denote the Heisenberg fan, given by the union of rays 
\[
R_{k}
:= \big\{(\lambda,(2k+n)|\lambda|)\,:\, \lambda \in \mathbb{R}^{\ast}\big\},
\qquad k = 0,1,2,\ldots,
\]
along with the limiting ray
\[
R_{\infty}
:= \big\{(0,\tau): \tau \ge 0 \big\}.
\]
According to \cite[Chapter 3]{thangaveluheisenberg}, the \(U(n)\)-spherical functions on \(\mathbb{H}^{n}\) are classified into the following two families:
\[
\frac{k!\,(n-1)!}{(k+n-1)!}
\, e^{\,i\lambda t}\,
\varphi^{\,n-1}_{k,\lambda}(z),
\qquad \text{and} \qquad
2^{\,n-1}(n-1)!
\,\frac{J_{n-1}(\sqrt{\tau}\,|z|)}{(\sqrt{\tau}\,|z|)^{n-1}},
\]
where \(J_{n-1}\) denotes the Bessel function of order \(n-1\), and 
\(\varphi^{\,n-1}_{k,\lambda}\) denotes the Laguerre function of type \(n-1\).

Using these spherical functions, Thangavelu \cite{Strichartz_Fourier_thangavelu} introduced the Strichartz Fourier transform. More precisely, for \(f\in L^{1}(\mathbb{H}^{n})\cap L^{2}(\mathbb{H}^{n})\), the Strichartz Fourier transform of \(f\) is defined on
\(\Omega \times \mathbb{C}^{n}\) as
\begin{equation}\label{sft def1}
    \widehat{f}(a,w)
    =
    \int_{\mathbb{H}^{n}}
        f(z,t)\,
        e_{a}\!\big((z,t)^{-1}(w,0)\big)
    \, dz\,dt,
\end{equation}
where \( a = (\lambda,(2k+n)|\lambda|) \in R_{k} \) and \( e_{a}(z,t) := e^{\,n-1}_{k,\lambda}(z,t)\). 

\noindent Furthermore, for an element \((0,\tau) \in R_{\infty}\), the Strichartz Fourier transform of \(f\) is given by  
\begin{equation*}
    \widehat{f}(0,\tau,w)
    = 2^{n-1}(n-1)! 
      \int_{\mathbb{H}^{n}} 
        f(z,t)\,
        \frac{J_{n-1}\!\left(\sqrt{\tau}\,|w-z|\right)}
             {(\sqrt{\tau}\,|w-z|)^{\,n-1}}\,dz\,dt.
\end{equation*}
We next recall the inversion formula and the Plancherel theorem for the Strichartz Fourier transform \cite{Strichartz_Fourier_thangavelu}. For this purpose, we first define a measure on the Heisenberg fan \(\Omega\). Since \(\Omega\) is a subset of \(\mathbb{R}^{2}\), it is naturally endowed with the induced Euclidean topology.
We define a measure \(\nu\) on the Heisenberg fan \(\Omega\) by
\begin{equation}\label{nu_def}
    \int_{\Omega} \phi(a)\, d\nu(a)
    := (2\pi)^{-2n-1}
       \int_{-\infty}^{\infty}
          \left(
             \sum_{k=0}^{\infty}
             \phi\big(\lambda,(2k+n)|\lambda|\big)
          \right)
          |\lambda|^{2n}\, d\lambda,
\end{equation}
for suitable functions \(\phi\) on \(\Omega\).

\begin{theorem}
For \(f \in \mathcal{S}(\mathbb{H}^n)\), the inversion formula for the Strichartz Fourier transform is given by
\begin{equation*}
f(z,t)
=
\int_{\Omega \times \mathbb{C}^n}
\widehat{f}(a,w)\,
e_a\!\left((-w,0)(z,t)\right)
\,dw\, d\nu(a).
\end{equation*}
Moreover, if \(f \in (L^1 \cap L^2)(\mathbb{H}^n)\), then the following Plancherel identity holds:

\begin{equation*}
    \int_{\mathbb{H}^n}|f(z,t)|^2 \, dz\,dt=\int_{\Omega \times \mathbb{C}^n}|\widehat{f}(a,w)|^2 \, dw \, d\nu(a).
\end{equation*}
\end{theorem}

The normalized Strichartz Fourier transform of \(f\), denoted 
\(\mathcal{F}(f)=\tilde{f}\), is defined as 
\begin{equation}\label{norm_sft}
    (\mathcal{F}(f))(a,w)
    = \tilde{f}(a,w)
    := c_{n,k}\,\widehat{f}(a,w)
    = \frac{k!\,(n-1)!}{(k+n-1)!}\,\widehat{f}(a,w),
\end{equation}
for \(a=(\lambda,(2k+n)|\lambda|)\in R_{k}\).  
For the element \((0,\tau)\in R_{\infty}\), we set 
\begin{equation*}\label{norm_sft_infty}
    (\mathcal{F}(f))(0,\tau,w)
    = \tilde{f}(0,\tau,w)
    := \widehat{f}(0,\tau,w).
\end{equation*}

\noindent Using the estimates for Laguerre and Bessel functions (see \cite{szego}), Thangavelu in
\cite{Strichartz_Fourier_thangavelu} proved that the normalized Strichartz Fourier transform satisfies 
\begin{equation}\label{L1 to L inf bdd}
  \sup_{(a,w)\in \Omega \times \mathbb{C}^{n}}
  \big|(\mathcal{F}(f))(a,w)\big|
  \leq C_n \|f\|_{1},
\end{equation}
for some constant \(C_n\).

\noindent The inversion formula can be stated in terms of the normalized Strichartz Fourier transform as
\begin{equation}\label{inversionformula}
   f(z,t)=\int_{\Omega \times \mathbb{C}^n}c_{n,k}\,(\mathcal{F}({f}))(a,w) \; e_a \left((-w,0)(z,t)\right) dw \, d\nu_2(a), 
\end{equation}
and the Plancherel identity takes the form
\begin{equation}\label{plancherel}
    \int_{\mathbb{H}^{n}} |f(z,t)|^{2}\,dz\,dt
    = \int_{\Omega \times \mathbb{C}^{n}}
      \big|(\mathcal{F}(f))(a,w)\big|^{2}\,dw\, d\nu_{2}(a),
\end{equation}
where \(d\nu_{2}(a)\) is a measure on \(\Omega\) given by  
\begin{equation}\label{measure nu_2}
\int_{\Omega} \phi(a)\, d\nu_{2}(a)
:= (2\pi)^{-2n-1}
   \int_{-\infty}^{\infty}
      \left(
           \sum_{k=0}^{\infty}
           \left(\frac{(k+n-1)!}{k!\,(n-1)!}\right)^{2}
           \phi(\lambda,(2k+n)|\lambda|)
      \right)
      |\lambda|^{2n}\,d\lambda.
\end{equation}

\noindent Throughout this paper, we use the notation \(\mathcal{F}(f)\) to denote the normalized 
Strichartz Fourier transform.

Unlike the classical Fourier transform on \(\mathbb{R}^{n}\) which is a unitary operator on 
\(L^{2}(\mathbb{R}^{n})\), the Strichartz Fourier transform does not define a unitary operator from 
\(L^{2}(\mathbb{H}^{n})\) onto 
\(L^{2}(\Omega \times \mathbb{C}^{n}, d\nu\,dw)\). It follows from \cite[(7.26)]{Strichartz_Fourier_thangavelu} that for any $f \in L^2(\mathbb{H}^n),$ the function $\widehat{f}(a,w)$ belongs to $L^2(\Omega \times \mathbb{C}^n, d\nu \, dw)$ and satisfies
\begin{equation*}
    (2\pi)^{-n}|\lambda|^{n}
    \left(
        \varphi^{\,n-1}_{k,\lambda} 
        \ast_{\lambda} 
        \widehat{f}(a,\cdot)
    \right)(z)
    = \widehat{f}(a,z),
    \qquad a \in R_{k}.
\end{equation*}
As a consequence, we restrict to the subspace $L^2_0(\Omega \times \mathbb{C}^n, d\nu \, dw)$ of $L^2(\Omega \times \mathbb{C}^n, d\nu \, dw)$ defined by
\begin{equation}
\left\{ F \in L^2(\Omega \times \mathbb{C}^n, d\nu \, dw) : (2\pi)^{-n}|\lambda|^{n}
    \left(
        \varphi^{\,n-1}_{k,\lambda} 
        \ast_{\lambda} 
        F(a,\cdot)
    \right)(z)
    = F(a,z)\, \text{for all} \; a \in R_{k} \right\},
\end{equation}
where  the twisted convolution $\ast_\lambda$ of two suitable functions $f_1$ and $f_2$ on $\mathbb{C}^n$ is defined as $$(f_1 \ast_\lambda f_2)(z) := \int_{\mathbb{C}^n} f_1(z - w) f_2(w) \, e^{\frac{i}{2} \lambda \operatorname{Im}(z\overline{w})} \, dw.$$
Considering the subspace $L^2_0(\Omega \times \mathbb{C}^n, d\nu \, dw)$, Thangavelu proved the following result in \cite{Strichartz_Fourier_thangavelu}.
\begin{theorem}
    The Strichartz Fourier transform is a unitary operator from \(L^2(\mathbb{H}^n)\) onto \(L_0^2(\Omega \times \mathbb{C}^n, d\nu\;dw)\).
\end{theorem}
One of the key ingredients in establishing boundedness of Fourier multipliers are the interpolation theorems. To make the paper self-contained we recall here mixed norm spaces and the Lorentz spaces and their associated interpolation theorems.
\subsection{Mixed norm spaces}
In this paper, we follow the terminology in \cite{Tapendu-multipliers}. The reader can refer to \cite{Benedek-mixed_norms,Huang-Survey_mixed_norms} for detailed study.

Let \((X_i,\mu_i)\) for \(i=1,2\) be two \(\sigma\)- finite measure spaces and \((X,\mu)\) be their product measure space. For an ordered pair \(p=(p_1,p_2)\), the mixed norm of a measurable function \(f\) on \((X,\mu)\) is defined as 
\begin{equation*}
    \lVert f \rVert_{(p_1,p_2)}=\left ( \int_{X_1} \left( \int_{X_2} |f(x_1,x_2)|^{p_1}\,dx_2\right)^{p_2/p_1}\,dx_1 \right)^{1/p_2}.
\end{equation*}
We denote by \(L^{(p_1,p_2)}(X_1 \times X_2)\) the space of all functions \(f\) satisfying
\[
\lVert f \rVert_{(p_1,p_2)} < \infty.
\]

\noindent Similar to the classical \(L^p\) spaces, mixed norm spaces also admit an interpolation theory. The following theorem is well known and can be found in \cite{Benedek-mixed_norms, Huang-Survey_mixed_norms}.
\begin{thm}
    Let \((X,\mu)\) and \((Y,\nu)\) be two mixed norm spaces. Let \(T\) be a linear map from \(X\) to \(Y\) such that 
    \begin{equation*}
        \lVert Tx \rVert_{(p^i_1,p^i_2)}\leq M_i \lVert x \rVert_{(q^i_1,q^i_2)}, \text{ for } i=1,2.
    \end{equation*}
    Then \(T\) also satisfies the inequality
    \begin{equation*}
        \lVert Tx \rVert_{(p^\theta_1,p^\theta_2)}\leq M_1^{1-\theta}M_2^\theta\lVert x \rVert_{(q^\theta_1,q^\theta_2)},
    \end{equation*}
    where \(0<\theta<1\) and 
    \begin{equation*}
        \frac{1}{p^\theta_j}=\frac{1-\theta}{p^1_j}+\frac{\theta}{p_j^2}, \text{ and } \frac{1}{q^\theta_j}=\frac{1-\theta}{q^1_j}+\frac{\theta}{q_j^2}.
    \end{equation*}
\end{thm}
\noindent Note that the order of integration taken in mixed norm spaces is different from that considered in \cite{Benedek-mixed_norms}, as we follow the terminology as in \cite{Tapendu-multipliers}.

\subsection{Lorentz spaces}
Next we recall the Lorentz spaces and the interpolation theorem associated to these spaces. Let \((X,\mu)\) be a measure space and \(f\) be a measurable function on \(X\). The distribution function of \(f\), denoted by \(d_f\) is defined on \([0,\infty)\) as 
$$
d_{f}(M)=\mu \left \{x: |f(x)|>M\right \}.
$$
\noindent The decreasing rearrangement of \(f\), denoted by \(f^\ast\), is a decreasing function on \([0,\infty)\), with the property that it shares the same distribution function as \(f\), that is,
 $$d_{f}=d_{f^*}.$$ 
Explicitly, for \(t \geq 0\), it is defined as
$$
f^*(t)=\inf \left \{s>0:d_{f}(s)\leq t\right \}.
$$
For details and further properties of decreasing rearrangement, we refer the reader to \cite{grafakos_book}.
\begin{defn}
    Let \(f\) be a measurable function on \((X,\mu)\), \(f^\ast\) denote its decreasing rearrangement \cite{grafakos_book} and \(0 <p,q \leq \infty\), and define
    $$
    \lVert f \rVert_{L^{p,q}}=\begin{cases}
    \left (\displaystyle\int\limits_{0}^{\infty}\left( t^{1/p}f^\ast(t)\right)^q\frac{dt}{t} \right)^{1/q}\; &\textrm{if}\; q< \infty \\
    \sup\limits_{t>0}t^{1/p}f^\ast(t)\; &\textrm{if}\; q=\infty.
\end{cases} 
    $$
    The set of all functions \(f\) for which \({\lVert f \rVert}_{L^{p,q}}< \infty\), denoted by \(L^{p,q}(X)\), is called the Lorentz space with indices \(p\) and \(q\).
\end{defn}
\noindent We now recall the Marcinkiewicz interpolation theorem for Lorentz spaces \cite[Theorem 4.13, Corollary 4.14]{bennett's_book}.

\begin{theorem}\label{Marcinkiewicz}
    Let \((X,\mu)\) and \((Y,\nu)\) be two measure spaces and let \(1 \leq p_1 < p_2 < \infty\) and \(1 \leq q_1,q_2 \leq \infty, \; q_1 \neq q_2,\; p_i\leq q_i,\; i=1,2\). Suppose T is a quasilinear operator on \(L^{p_1,1}(X,\mu)+L^{p_2,1}(X,\mu)\) taking values in the space of measurable functions on \((Y,\nu)\), such that \(T\) is simultaneously weak type \((p_1,q_1)\) and \((p_2,q_2)\), then for any \(0<t<1\), and 
    $$
    \frac{1}{p}=\frac{1-t}{p_1}+\frac{t}{p_2},\; \frac{1}{q}=\frac{1-t}{q_1}+\frac{t}{q_2},
    $$
    the operator \(T\) is bounded from \(L^p\) to \(L^q\).
\end{theorem}
\noindent Note that in the above theorem, by weak type \((p_1,q_1)\), we mean that the the operator \(T\) is bounded from \(L^{p_1,1} \text{ to } L^{q_1,\infty}\).

\noindent Throughout the paper we will use the notation 
\(a=(\lambda,(2k+n)|\lambda|)\), where \(\lambda\in \mathbb{R}^\ast=\mathbb{R}\setminus\{0\}\) and 
\(k\in \mathbb{N}\cup\{0\}\). Moreover, \(L^{(p,q)}\) will denote the mixed norm 
space, whereas \(L^{p,q}\) denotes the Lorentz space. We adopt the standard convention of using \(C_1,C_2\) and \(C_p\) to denote positive constants whose values may vary from one equation to the other. 

 We next discuss the boundedness of the Fourier multiplier operator. For a bounded function \(m\) on \(\Omega\), consider the operator \(T_m\) on \(\mathcal{S}(\mathbb{H}^n)\) defined as
\begin{equation}\label{operator T_m}
    \mathcal{F}\left({T_m(f)}\right)(a,w)=m(a)\mathcal{F}(f)(a,w).
\end{equation}
\begin{remark}
    Observe that we are considering \(m\) as a function of \(a \in \Omega\) and not a function of \((a,w)\in \Omega \times \mathbb{C}^n\) because of the complexity in considering the inverse image of the function \(m(a,w)\mathcal{F}(a,w)\) under the Strichartz Fourier transform.
\end{remark}
Consider left translations on \(\mathbb{H}^n\) defined as 
\begin{equation*}
    \tau_{(u,s)}f(z,t)=f\left((u,s)^{-1}(z,t) \right).
\end{equation*}
Recall that Fourier multiplier operators on \(\mathbb{R}^n\) commutes with translations. We prove a similar result in the setting of the Heisenberg group.
\begin{lemma}\label{commutativity of translation and T_m}
Let \((u,s) \in \mathbb{H}^n\), and let \(T_m\) be the Fourier multiplier operator defined above. Then, for every \(f \in \mathcal{S}(\mathbb{H}^n)\), we have
\begin{equation*}
  \tau_{(u,s)}\left(T_m(f)\right)(z,t)=T_m\left(\tau_{(u,s)} f\right)(z,t).  
\end{equation*}
\end{lemma}
\begin{proof}
We will show that 
$$
   \mathcal{F}\left(\tau_{(u,s)}T_m(f) \right)(a,w)=\mathcal{F}\left(T_m\tau_{(u,s)}(f) \right)(a,w).
$$
Observe that for any \(g \in \mathcal{S}(\mathbb{H}^n)\),
\begin{align*}
    \mathcal{F}\left(\tau_{(u,s)}g \right)(a,w)&=c_{n,k}\int_{\mathbb{H}^n}\tau_{(u,s)}g(z,t) e_a\left((z,t)^{-1}(w,0) \right)\,dz\,dt\\
&=c_{n,k}\int_{\mathbb{H}^n}g\left((u,s)^{-1}(z,t) \right)e_a\left((z,t)^{-1}(w,0) \right)\,dz\,dt\\
    &=c_{n,k}\int_{\mathbb{H}^n}g(z,t) e_a\left((z,t)^{-1}(u,s)^{-1}(w,0) \right)\,dz\,dt\\
    &=e^{-i\lambda\left(s+\frac{1}{2}\operatorname{Im}(u \bar{w})\right)}c_{n,k}\int_{\mathbb{H}^n}g(z,t) e_a\left((z,t)^{-1}(-u+w,0) \right)\,dz\,dt\\
    &=e^{-i\lambda\left(s+\frac{1}{2}\operatorname{Im}(u \bar{w})\right)}\mathcal{F}(g)(a,-u+w).
\end{align*}
Using the above computation, we see that 
\begin{align}
    \mathcal{F}\left(\tau_{(u,s)}T_m(f) \right)(a,w)&=e^{-i\lambda\left(s+\frac{1}{2}\operatorname{Im}(u \bar{w})\right)}\mathcal{F}\left(T_m(f)\right)(a,w)\nonumber\\
    &=e^{-i\lambda\left(s+\frac{1}{2}\operatorname{Im}(u \bar{w})\right)}m(a)\mathcal{F}(f)(a,w)\nonumber\\
    &=m(a)\mathcal{F}\left(\tau_{(u,s)}(f)\right)(a,w)\nonumber \\
    &=\mathcal{F}\left(T_m \tau_{(u,s)}(f)\right)(a,w).
\end{align}
This completes the proof.
\end{proof}
\noindent An immediate consequence of this theorem is that the boundedness of 
\(T_m\) from \(L^p(\mathbb{H}^n) \) to \(L^q(\mathbb{H}^n)\)
 forces the relation \(p\leq q\).
\begin{theorem}
    Let \(T_m\) be bounded from \(L^p(\mathbb{H}^n)\) to \(L^q(\mathbb{H}^n)\). Then we must have \(p \leq q\).
\end{theorem}
\begin{proof}
 We proceed as in the proof of \cite[Theorem 1.1]{hormander-estimates_of_translation}. 
Suppose, to the contrary, that $p>q$ and $T_m$ is bounded from $L^p(\mathbb{H}^n)$ to $ L^q(\mathbb{H}^n)$. 
Then, for $f\in L^p(\mathbb{H}^n)$,
\begin{equation}\label{translated norm}
\|f+\tau_{(u,s)}f\|_p \to 2^{1/p}\|f\|_p 
\quad \text{as } |(u,s)|\to\infty .
\end{equation}
Using the boundedness of \(T_m\) together with Lemma \ref{commutativity of translation and T_m}, we obtain 
\begin{align*}
    \lVert T_mf+\tau_{(u,s)}T_mf \rVert_q&=\lVert T_m(f+\tau_{(u,s)}f)\rVert_q\\
    &\leq \lVert T_m \rVert \lVert f+\tau_{(u,s)}f \rVert_p.
\end{align*}
Taking the limit \(|(u,s)| \to \infty\), and applying \eqref{translated norm} in the above inequality, we get 
\begin{equation*}
    2^{1/q}\lVert T_mf \rVert_q \leq 2^{1/p}\lVert T_m \rVert \lVert f \rVert_p.
\end{equation*}
Hence,
\[\lVert T_mf \rVert_q \leq 2^{1/p-1/q}\lVert T_m \rVert \lVert f \rVert_p,\]
Since \(p>q\), we have \(2^{1/p-1/q}<1\), and therefore
\[
2^{1/p-1/q}\lVert T_m \rVert < \lVert T_m \rVert,
\]
which contradicts the definition of the operator norm of \(T_m.\).

\end{proof}
As mentioned in the introduction, for \(\mathbb{R}^n\), there are two results on the boundedness of the Fourier multiplier operator. One is the classical H\"ormander's multiplier \cite{hormander-estimates_of_translation} theorem which concerns the case \(1<p\leq 2\leq q<\infty\). The other is the Lizorkin theorem \cite{Lizorkin} which applies for \(1 < p \leq q \leq 2\)
and \(2 \leq p \leq q<\infty\), however, it imposes regularity conditions on the symbol. In this paper, we will establish the H\"ormander's multiplier theorem.

\section{\(L^p-L^q\) boundedness}\label{sec3}
In this section, we prove the \(L^p\)-\(L^q\) boundedness of the operator \(T_m\) in the range \(1<p\leq 2\leq q<\infty\). As a preliminary step, we establish a Hausdorff-Young inequality for the Strichartz Fourier transform corresponding to a modified measure on the Heisenberg fan \(\Omega\).
\noindent Let \(\mu\) be the measure on \(\Omega\) defined by
\begin{equation}
    d\mu(a)=|\lambda|^{-n}\,d\nu_2(a),
\end{equation}
where \(\nu_2\) is given by \eqref{measure nu_2}.

\begin{theorem}\label{H-Y inequality theorem}
   Let \(1\leq p\leq 2\) and \(f\in \mathcal{S}(\mathbb{H}^n)\). Then
\begin{equation}\label{new h-y p<2}
\left(\int_{\Omega}
\left(\int_{\mathbb{C}^n}
\lvert \mathcal{F}(f)(a,w)\rvert^2\,dw
\right)^{p'/2}
|\lambda|^{np'/2}\,d\mu(a)
\right)^{1/p'}
\leq C_p \lVert f\rVert_p.
\end{equation}

Furthermore, if \(p>2\) and \(f\in \mathcal{S}(\mathbb{H}^n)\), then
\begin{equation}\label{h-y p>2}
\lVert f\rVert_p
\leq C_p
\left(
\int_{\Omega}
\left(
\int_{\mathbb{C}^n}
\lvert \mathcal{F}(f)(a,w)\rvert^2\,dw
\right)^{p'/2}
|\lambda|^{np'/2}\,d\mu(a)
\right)^{1/p'}.
\end{equation}
\end{theorem}
\begin{proof}
   Consider the map \(H\), defined on the set of simple functions on \(\mathbb{H}^n\), by
\[
H(f)(\lambda,(2k+n)|\lambda|,w)
=
|\lambda|^{n/2}\mathcal{F}(f)(\lambda,(2k+n)|\lambda|,w),
\]
which takes values in the space of measurable functions on \((\Omega\times \mathbb{C}^n,d\mu\,dw)\).
   \noindent We first show that 
    \begin{equation}\label{L1 L2,inf bdd}
      \lVert H(f)\rVert_{(2,\infty)} \leq (2\pi)^{n/2}\lVert f \rVert_1.  
    \end{equation}
\noindent Let \(a=(\lambda,(2k+n)|\lambda|)\in \Omega\). Then
\begin{align*}
\left(
\int_{\mathbb{C}^n}
|H(f)(a,w)|^2\,dw
\right)^{1/2}
&=
|\lambda|^{n/2}
\left(
\int_{\mathbb{C}^n}
|\mathcal{F}(f)(a,w)|^2\,dw
\right)^{1/2}\\
&=
|\lambda|^{n/2}
\left(
\int_{\mathbb{C}^n}
\left|
\int_{\mathbb{H}^n}
c_{n,k}f(z,t)
e_a\big((z,t)^{-1}(w,0)\big)\,dz\,dt
\right|^2
dw
\right)^{1/2}\\
&\leq
|\lambda|^{n/2}c_{n,k}
\int_{\mathbb{H}^n}
\left(
\int_{\mathbb{C}^n}
|f(z,t)|^2
|\varphi^{n-1}_{k,\lambda}(-z+w)|^2\,dw
\right)^{1/2}
dz\,dt\\
&\leq
(2\pi)^{n/2}\|f\|_1.
\end{align*}
Since the above estimate holds for every \(a=(\lambda,(2k+n)|\lambda|)\in \Omega\), it follows that \eqref{L1 L2,inf bdd} holds.
    \noindent Next, applying the Plancherel theorem, we get 
    \begin{align}
        \lVert H(f) \rVert_{(2,2)}&=\left( \int_{\Omega} \int_{\mathbb{C}^n}  |\lambda|^n \lvert \mathcal{F}(f)(a,w) \rvert^2\,dw\,d\mu(a) \right)^{1/2}\nonumber\\
        &=\left( \int_{\Omega} \int_{\mathbb{C}^n}   \lvert \mathcal{F}(f)(a,w) \rvert^2\,dw\,d\nu_2(a) \right)^{1/2}\nonumber\\
        &=\lVert f \rVert_2.\label{L2 L2 bdd}
    \end{align}
    By using \eqref{L1 L2,inf bdd} and \eqref{L2 L2 bdd} in interpolation of mixed norm spaces, we deduce \eqref{new h-y p<2}.

    To establish \eqref{h-y p>2}, we use duality argument. Let \(p>2\) then we have 
    \begin{align*}
        \lVert f \rVert_p&=\sup_{\lVert g \rVert_{p'} \leq 1} \left \lvert \int_{\mathbb{H}^n} f(z,t) \overline{g(z,t)}\,dz\,dt \right \rvert\\
        &=\sup_{\lVert g \rVert_{p'} \leq 1}\left \lvert  \int_{\Omega} \int_{\mathbb{C}^n} \mathcal{F}(f)(a,w) \overline{\mathcal{F}(g)(a,w)}\,dw\,d\nu_2(a) \right \rvert\\
        &\leq \sup_{\lVert g \rVert_{p'} \leq 1} \left( \int_{\Omega} \left( \int_{\mathbb{C}^n} \left| \mathcal{F}(f)(a,w)\right|^2 \,dw\right)^{p'/2} |\lambda|^{np'/2-n}\,d\nu_2(a)  \right)^{1/p'} \\
        &\hspace{6cm}\left( \int_{\Omega} \left( \int_{\mathbb{C}^n} \left| \mathcal{F}(g)(a,w)\right|^2 \,dw\right)^{p/2} |\lambda|^{np/2-n}\,d\nu_2(a) \right)^{1/p}\\
        &= \sup_{\lVert g \rVert_{p'} \leq 1} \left( \int_{\Omega} \left( \int_{\mathbb{C}^n} \left| \mathcal{F}(f)(a,w)\right|^2 \,dw\right)^{p'/2} |\lambda|^{np'/2}\,d\mu(a)  \right)^{1/p'} \\
        &\hspace{6cm}\left( \int_{\Omega} \left( \int_{\mathbb{C}^n} \left| \mathcal{F}(g)(a,w)\right|^2 \,dw\right)^{p/2} |\lambda|^{np/2}\,d\mu(a) \right)^{1/p}\\
        &\leq C_p \sup_{\lVert g \rVert_{p'} \leq 1} \lVert g \rVert_{p'}\left( \int_{\Omega} \left( \int_{\mathbb{C}^n} \left| \mathcal{F}(f)(a,w)\right|^2 \,dw\right)^{p'/2} |\lambda|^{np'/2}\,d\mu(a)  \right)^{1/p'}\\
        &\leq C_p \left( \int_{\Omega} \left( \int_{\mathbb{C}^n} \left| \mathcal{F}(f)(a,w)\right|^2 \,dw\right)^{p'/2} |\lambda|^{np'/2}\,d\mu(a)  \right)^{1/p'}.
    \end{align*}
    This completes the proof.
\end{proof}
The following theorem establishes a Paley inequality for the Strichartz Fourier transform. It yields the boundedness of the Strichartz Fourier transform from \(L^p(\mathbb{H}^n)\) to \(L^{(2,p)}(\Omega \times \mathbb{C}^n)\), equipped with an appropriate measure. The proof uses the interpolation result stated in Theorem \ref{Marcinkiewicz} from Section \ref{sec2}.
\begin{theorem}\label{Paley's inequality theorem}
   Let \(1<p\leq 2\), and let \(\phi\in L^{1,\infty}(\Omega,d\mu(a))\) be a positive function. Then, for every \(f\in \mathcal{S}(\mathbb{H}^n)\), we have
\[
\left(
\int_{\Omega}
\left(
\int_{\mathbb{C}^n}
\left|\mathcal{F}(f)(a,w)\right|^2\,dw
\right)^{p/2}
|\lambda|^{np/2}\phi^{2-p}(a)\,d\mu(a)
\right)^{1/p}
\leq
C_p
\bigl(\|\phi\|_{1,\infty}\bigr)^{2/p-1}
\|f\|_p.
\]
\end{theorem}
\begin{proof}
   Consider the map \(T\), defined on \(L^{1,1}(\mathbb{H}^n)+L^{2,1}(\mathbb{H}^n)\) and taking values in the space of measurable functions on \((\Omega,\phi^2(a)\,d\mu(a))\), by
\[
T(f)(a)=
\begin{cases}
\displaystyle
\frac{|\lambda|^{n/2}}
{\phi(\lambda,(2k+n)|\lambda|)}
\left\|
\mathcal{F}(f)(\lambda,(2k+n)|\lambda|,\cdot)
\right\|_2,
& \text{if } a=(\lambda,(2k+n)|\lambda|)\in R_k,\\[1.2em]
0,
& \text{if } a=(0,\tau)\in R_{\infty}.
\end{cases}
\]
Observe that for \(f\in L^2(\mathbb{H}^n)\), the Plancherel theorem yields
    \begin{align*}
        \lVert T(f) \rVert_2^2&=\int_{\Omega}\frac{|\lambda|^n}{\phi^2(\lambda,(2k+n)|\lambda|)}\left(\int_{\mathbb{C}^n}|\mathcal{F}(F)(\lambda,(2k+n)|\lambda|,w)|^2\,dw \right)\phi^2(\lambda,(2k+n)|\lambda|)\,d\mu(a)\\
        &=\lVert f \rVert_2^2.
    \end{align*}
   Hence, the operator \(T\) is of strong type \((2,2)\), and therefore also of weak type \((2,2)\).

\noindent Next, we prove that \(T\) is of weak type \((1,1)\). Equivalently, it suffices to show that
\begin{equation}\label{equiv weak 1,1}
\left| \{a\in \Omega: |Tf(a)|>s \} \right|_{\phi^2}
\leq
C\|\phi\|_{1,\infty}\frac{\|f\|_1}{s},
\qquad s>0.
\end{equation}
Here, \(|\cdot|_{\phi^2}\) denotes the weighted measure on \(\Omega\) associated with the weight \(\phi^2(a)\).

Recalling the hypothesis, \(\phi\in L^{1,\infty}(\Omega)\). Hence, there exists a constant \(C\) such that
\begin{equation}\label{phi-in-L1inf}
\left|
\{a\in \Omega:\phi(a)>s\}
\right|
\leq
\frac{C}{s}.
\end{equation}
Note that, from the proof of Theorem \ref{H-Y inequality theorem}, we have
\begin{multline*}
\left|
\left\{
(\lambda,(2k+n)|\lambda|):
\frac{|\lambda|^{n/2}}
{\phi(\lambda,(2k+n)|\lambda|)}
\left\|
\mathcal{F}(f)(\lambda,(2k+n)|\lambda|,\cdot)
\right\|_2
>s
\right\}
\right|_{\phi^2}
\\
\leq
\left|
\left\{
(\lambda,(2k+n)|\lambda|):
\frac{\|f\|_1}
{\phi(\lambda,(2k+n)|\lambda|)}
>(2\pi)^{-n/2}s
\right\}
\right|_{\phi^2}.
\end{multline*}
Therefore, to establish \eqref{equiv weak 1,1}, it suffices to prove that
\begin{equation}\label{claim for weak type}
\left|
\left\{
(\lambda,(2k+n)|\lambda|):
\frac{\|f\|_1}
{\phi(\lambda,(2k+n)|\lambda|)}
>
(2\pi)^{-n/2}s
\right\}
\right|_{\phi^2}
\leq
2C\frac{\|f\|_1}{s},
\end{equation}
where the constant \(C\) is independent of \(s\).

\noindent Assume that \(\|f\|_1\neq 0\), and set
\[
\gamma=(2\pi)^{-n/2}\frac{s}{\|f\|_1}.
\]
Using the layer cake representation theorem \cite{Lieb-Loss}, we obtain
\begin{align*}
&\left|
\left\{
(\lambda,(2k+n)|\lambda|):
\frac{1}{\phi(\lambda,(2k+n)|\lambda|)}
>\gamma
\right\}
\right|_{\phi^2}
\\
&=
\int_0^\infty
2t
\left|
\left\{
(\lambda,(2k+n)|\lambda|):
t<\phi(\lambda,(2k+n)|\lambda|)
<\frac{1}{\gamma}
\right\}
\right|
dt
\\
&=
\int_0^{1/\gamma}
2t\bigl(m(t)-m(1/\gamma)\bigr)\,dt,
\end{align*}
where
\[
m(t)
=
\left|
\left\{
(\lambda,(2k+n)|\lambda|):
\phi(\lambda,(2k+n)|\lambda|)>t
\right\}
\right|.
\]
Therefore,
\begin{align*}
\left|
\left\{
(\lambda,(2k+n)|\lambda|):
\frac{1}{\phi(\lambda,(2k+n)|\lambda|)}
>\gamma
\right\}
\right|_{\phi^2}
&=
\int_0^{1/\gamma}
2t\,m(t)\,dt
-\frac{1}{\gamma^2}m(1/\gamma)
\\
&\leq
\int_0^{1/\gamma}
2t\,m(t)\,dt
\\
&\leq
2\frac{C}{\gamma}.
\end{align*}
Substituting this estimate into \eqref{claim for weak type}, we conclude that \(T\) is of weak type \((1,1)\). The desired result now follows from the Marcinkiewicz interpolation theorem.
 \end{proof}
Combining the previous two theorems, we obtain the following Hausdorff-Young-Paley inequality. Its proof relies on an interpolation theorem for mixed norm spaces involving a change of measure. The following result follows from \cite[Corollary 7.3]{Tapendu-multipliers}.
\begin{lemma}\label{interpolation of mixed spaces weights}
    Let \(X=X_1 \times X_2\) and \(Y=Y_1 \times Y_2\) be two mixed norm spaces with product measures \(dx\) and \(dy\) respectively, and \(T\) is a linear map from \(X\) to \(Y\). Suppose \(w_j:Y_1 \to \mathbb{R}^+\) be non negative weight functions which are integrable on sets of finite measure, for \(j=1,2\). If for all simple functions \(f\) on X, 
    $$
    \lVert T(f) \rVert_{L^{(p_1^j,p_2^j)}(Y,w_jdy)} \leq M_j \lVert f \rVert_{L^{(q,q)}(X,dx)}, \text{ for } j=1,2.
    $$
   Then 
   $$
   \lVert T(f) \rVert_{L^{(p_1^\theta,p_2^\theta)}(Y,w_\theta dy)} \leq M_1^{1-\theta}M_2^\theta \lVert f \rVert_{L^{(q,q)(X,dx)}},
   $$
   where \(0<\theta<1\) and 
    \begin{equation*}
        \frac{1}{p^\theta_j}=\frac{1-\theta}{p^1_j}+\frac{\theta}{p_j^2}, \text{ and } \frac{1}{q^\theta_j}=\frac{1-\theta}{q^1_j}+\frac{\theta}{q_j^2},
    \end{equation*}
    and 
    $$
    w_\theta=w_1^{\frac{p^\theta_2(1-\theta)}{p_2^1}} w_2^{\frac{p^\theta_2 \theta}{p_2^2}}.
    $$
\end{lemma}

\noindent We now state and prove the Hausdorff-Young-Paley inequality in our framework.
\begin{theorem}\label{hausdorff-young-paley}
    Let \(1< p \leq 2\), \(1<p\leq b \leq p'<\infty\) and \(\phi \in L^{1,\infty}(\Omega,d\mu(a))\). Then 
    \begin{equation*}
        \left( \int_{\Omega}\left( \int_{\mathbb{C}^n} |\lambda|^n|\mathcal{F}(f)(a,w)|^2\,dw \right)^{b/2}\phi^{1-b/p'}\,d\mu(a)\right)^{1/b}\leq C_b(\lVert \phi \rVert_{1,\infty})^{1/b-1/p'} \lVert f \rVert_p.
    \end{equation*}
\end{theorem}
\begin{proof}
   Let \(1<p\leq 2\), and define the operator \(T\) on \(L^p(\mathbb{H}^n)\), taking values in the space of measurable functions on \(\Omega\times\mathbb{C}^n\), by
\[
T(f)(a,w)=
\begin{cases}
|\lambda|^{n/2}\mathcal{F}(f)(\lambda,(2k+n)|\lambda|,w),
& \text{if } a=(\lambda,(2k+n)|\lambda|)\in R_k,\\
0,
& \text{if } a=(0,\tau)\in R_{\infty}.
\end{cases}
\]
From \eqref{new h-y p<2}, it follows that \(T\) is bounded from \(L^p(\mathbb{H}^n)\) to \(L^{(2,p')}(\Omega\times\mathbb{C}^n,d\mu\,dw)\). Further, by Theorem \ref{Paley's inequality theorem}, \(T\) is bounded from \(L^p(\mathbb{H}^n)\) to \(L^{(2,p)}(\Omega\times\mathbb{C}^n,\phi^{2-p}\,d\mu\,dw)\).

Moreover, it is straightforward to verify that \(\phi^{2-p}\) is locally integrable on \((\Omega,d\mu)\). Therefore, applying Lemma \ref{interpolation of mixed spaces weights} and proceeding as in the proof of \cite[Theorem 3.1]{Tapendu-multipliers}, we obtain that
    \begin{equation*}
        \left( \int_{\Omega}\left( \int_{\mathbb{C}^n}|\lambda|^n |\mathcal{F}(f)(a,w)|^2\,dw  \right)^{b/2}\phi^{1-b/p'}\,d\mu(a)\right)^{1/b} \leq C_b (\lVert \phi_{1,\infty}\rVert)^{1/b-1/p'}\lVert f \rVert_p.
    \end{equation*}
    This completes the proof.
\end{proof}
\noindent We now establish the \(L^p\)-\(L^q\) boundedness of the multiplier operator \(T_m\).

\begin{theorem}\label{Lp Lq boundedness theorem}
Let \(1<p\leq 2\leq q<\infty\), and let \(m\) be a bounded function on \(\Omega\) satisfying
\begin{equation}\label{condition on m}
\sup_{\alpha>0}
\alpha
\left(
\int_{\{a\in \Omega:\,|m(a)|>\alpha\}}
d\mu(a)
\right)^{1/p-1/q}
<\infty.
\end{equation}
Then the operator \(T_m\), defined by \eqref{operator T_m}, extends to a bounded operator from \(L^p(\mathbb{H}^n)\) to \(L^q(\mathbb{H}^n)\).
\end{theorem}

\begin{proof}
   It suffices to consider the case \(p\leq q'\), since when \(p\geq q'\), one may pass to the adjoint operator \(T_m^\ast=T_{\overline{m}}\) and thereby reduce the problem to the case \(p\leq q'\).
    
\noindent    By Hausdorff-Young inequality, we have 
    \begin{equation}\label{bound of T_mf_q}
        \lVert T_mf \rVert_q \leq C_q \left(\int_{\Omega} \left(\int_{\mathbb{C}^n}\lvert \mathcal{F}(f)(a,w) \rvert^2\,dw \right)^{q'/2} m(a)^{q'}|\lambda|^{nq'/2}\,d\mu(a)\right)^{1/q'}.
    \end{equation}
    Taking \(\phi=|m|^r\), where \(\frac{1}{r}=\frac{1}{p}-\frac{1}{q},\; b=q'\) in Theorem \ref{hausdorff-young-paley}, we get 
    \begin{equation*}
        \left( \int_{\Omega}\left( \int_{\mathbb{C}^n} |\lambda|^n|\mathcal{F}(f)(a,w)|^2\,dw \right)^{q'/2}|m|^{r(1-q'/p')}\,d\mu(a)\right)^{1/q'}\leq C_p(\lVert m^r \rVert_{1,\infty})^{1/q'-1/p'} \lVert f \rVert_p.
    \end{equation*}
    Simplifying the above expression, we get 
    \begin{equation}\label{simplying H-Y Paley}
        \left( \int_{\Omega}\left( \int_{\mathbb{C}^n} |\mathcal{F}(f)(a,w)|^2\,dw \right)^{q'/2}|m|^{q'}|\lambda|^{nq'/2}\,d\mu(a)\right)^{1/q'}\leq C_p(\lVert m^r \rVert_{1,\infty})^{1/q'-1/p'} \lVert f \rVert_p.
    \end{equation}
    Furthermore, by definition, we have 
    \begin{align*}
        (\lVert m^r \rVert_{1,\infty})^{1/q'-1/p'}&=(\lVert m^r \rVert_{1,\infty})^{1/r}\\
        &=\left(\sup_{\alpha>0} \alpha^r \int_{\{a \in \Omega:|m(a)|>\alpha\}} d\mu(a) \right)^{1/r}\\
        &=\sup_{\alpha>0} \alpha \left(\int_{\{a\in \Omega:|m(a)|>\alpha\}} d\mu(a) \right)^{1/r}\\
        &=\sup_{\alpha>0} \alpha \left(\int_{\{a\in \Omega:|m(a)|>\alpha\}} d\mu(a) \right)^{1/p-1/q.}
    \end{align*}
   Combining this estimate with \eqref{bound of T_mf_q} and \eqref{simplying H-Y Paley}, we obtain the desired conclusion.
\end{proof}
\begin{corollary}
    Let \(\beta, \gamma >0\) and \(m\) be a measurable function on \(\Omega\) which is identically \(0\) on the ray \(R_\infty\) and on other rays is defined as
    \begin{equation*}
        m(\lambda,(2k+n)|\lambda|)= \begin{cases}
           (1+(2k+n)|\lambda|)^{-\gamma}(1+k)^{-\beta} & \text{ if } k>0,\\
           0 & \text{if } k=0. 
        \end{cases} 
    \end{equation*}
    Then \(T_m\) defines \(L^p-L^q\) bounded Fourier multiplier if 
    \begin{equation*}
        n-2<\frac{\beta}{\gamma}(n+1) \text{ \,\, and \,\,} \frac{n+1}{\gamma}\left(\frac{1}{p}-\frac{1}{q} \right)\leq1.
    \end{equation*}
\end{corollary}
\begin{proof}
   Observe that \(m\leq 1\). Therefore, it suffices to show that, under the assumptions of the theorem,
\begin{equation}
\sup_{0<\alpha<1}
\alpha
\left(
\int_{\{a\in \Omega:\,|m(a)|>\alpha\}}
d\mu(a)
\right)^{1/p-1/q}
<\infty.
\end{equation}

Let \(0<\alpha<1\). Note that if
\[
(1+(2k+n)|\lambda|)^{-\gamma}(1+k)^{-\beta}
>\alpha,
\]
then
\[
|\lambda|
<
\frac{\alpha^{-1/\gamma}}
{(2k+n)(1+k)^{\beta/\gamma}}.
\]
Using this estimate, we compute
    \begin{align*}
      \int_{\{a: |m(a)|>\alpha\}}d\mu(a)&= (2\pi)^{-2n-1} \sum_{k>1}(c_{n,k})^{-2} \int_{0}^{\frac{\alpha^{-1/\gamma}}{(2k+n)(1+k)^{\beta/\gamma}}}|\lambda|^n\,d\lambda\\
      &=(2\pi)^{-2n-1} 2 \sum_{k>1} \frac{(k+n-1)^{2(n-1)}}{(n+1)!^2(n+1)} \left(\frac{\alpha^{-\frac{1}{\gamma}}}{(2k+n)(1+k)^{\frac{\beta}{\gamma}}} \right)^{n+1},
    \end{align*}
    which is finite if and only if 
    $$
    n-3-\frac{\beta}{\gamma}(n+1)<-1.
    $$
    Subsequently, we are left with
    $$
    \sup_{0<\alpha<1} \alpha \left(\alpha^{-\frac{n+1}{\gamma}}\right)^{\frac{1}{p}-\frac{1}{q}}.
    $$
    The above quantity if finite if and only if 
    $$
    1-\frac{(n+1)}{\gamma}\left( \frac{1}{p}-\frac{1}{q} \right)\geq 0.
    $$
\end{proof}

\section{Application in PDEs}\label{sec4}
In this section, we employ the \(L^p\)-\(L^q\) boundedness of the operator \(T_m\) to study the well-posedness of certain nonlinear equations. Our approach is motivated by the works \cite{cardona-vishvesh-ruzhansky-multipliers,Vishvesh-ruzhsnaky-k-a_multipliers,shaimardan-tulenov-noncomm_torus}.

\noindent Throughout this section, \(\tau\) denotes the time variable.\\

\subsection{Nonlinear heat equations}

Consider the following two Cauchy problems in the space,\\ \(L^\infty(0,T;L^2(\mathbb{H}^n))\):
\begin{equation}\label{first heat equation}
u_{\tau}=|T_mu|^p,\qquad u(0)=u_0,
\end{equation}
and
\begin{equation}\label{second heat equation}
u_{\tau}+\mathcal{L}u+du=|T_mu|^p,\qquad u(0)=u_0,
\end{equation}
where \(T_m\) denotes the Fourier multiplier operator on \(L^2(\mathbb{H}^n)\) associated with a bounded function \(m\), \(1<p<\infty\), \(\mathcal{L}\) is the sublaplacian on the Heisenberg group, and \(d>0\).

A function \(u\in L^\infty(0,T;L^2(\mathbb{H}^n))\) is called a global solution of \eqref{first heat equation} if
\begin{equation}\label{solution of first heat equation}
u=u_0+\int_0^\tau |T_mu(s)|^p\,ds
\end{equation}
holds for every \(T<\infty\). Likewise, \(u\) is called a global solution of \eqref{second heat equation} if
\begin{equation}\label{solution of second heat equation}
u=e^{-\tau(\mathcal{L}+dI)}u_0
+\int_0^\tau e^{-(\tau-s)(\mathcal{L}+dI)}
|T_mu(s)|^p\,ds
\end{equation}
holds for every \(T<\infty\).

In either case, \(u\) is called a local solution if the corresponding integral equation holds for some \(T^\ast>0\).
\begin{theorem}\label{theorem for heat equations}
Let \(1<p<\infty\), and let \(T_m\) be the Fourier multiplier operator associated with a function \(m\) satisfying
\[
\sup_{\alpha>0}
\alpha
\left(
\int_{\{a\in \Omega:\,|m(a)|>\alpha\}}
d\mu(a)
\right)^{\frac{1}{2}-\frac{1}{2p}}
<\infty.
\]
Then the Cauchy problems \eqref{first heat equation} and \eqref{second heat equation} admit local solutions.
\end{theorem}
\begin{proof}
   We prove the existence of a local solution for \eqref{second heat equation}.

\noindent By Duhamel's principle, a solution of \eqref{second heat equation} is given by
\begin{equation}\label{form of general solution of second heat eqn}
\mathcal{K}(u)(\tau)
=
e^{-\tau(\mathcal{L}+dI)}u_0
+
\int_0^\tau
e^{-(\tau-s)(\mathcal{L}+dI)}
|T_mu(s)|^p\,ds.
\end{equation}

By the Hille--Yosida theorem \cite{semigroups-Goldstein}, the operator \(\mathcal{L}+dI\) generates a contractive semigroup on \(L^2(\mathbb{H}^n)\); see also \cite{Folland-subelliptic_estimates}.
Consequently, applying Minkowski's integral inequality in \eqref{form of general solution of second heat eqn}, we get 
   \begin{align}
       \left\lVert \int_{0}^\tau e^{-(\tau-s)(\mathcal{L}+dI)}|T_mu(s)|^p\,ds\right\rVert_{L^2(\mathbb{H}^n)} & \leq  \int_{0}^\tau  \left\lVert e^{-(\tau-s)(\mathcal{L}+dI)}\left(|T_mu(s)|^p \right)\right\rVert_{L^2(\mathbb{H}^n)}\,ds\nonumber\\
       &\leq  \int_{0}^\tau  \left\lVert |T_mu(s)|^p \right\rVert_{L^2(\mathbb{H}^n)}\,ds\nonumber\\
       &= \int_{0}^\tau  \left\lVert T_mu(s)\right\rVert_{L^{2p}(\mathbb{H}^n)}^p\,ds\nonumber\\
       &\leq C \int_{0}^\tau  \left\lVert u(s)\right\rVert_{L^2(\mathbb{H}^n)}^p\,ds, \label{estimate in solution of second heat eqn}
   \end{align}
  where, in the last inequality, we used the fact that \(T_m:L^2(\mathbb{H}^n)\to L^{2p}(\mathbb{H}^n)\) is bounded under the given hypothesis. Combining \eqref{form of general solution of second heat eqn} and \eqref{estimate in solution of second heat eqn}, we obtain
\[
\|\mathcal{K}(u)\|_{L^\infty(0,T;L^2(\mathbb{H}^n))}
\leq
\|u_0\|_{L^2(\mathbb{H}^n)}
+
CT
\|u\|_{L^\infty(0,T;L^2(\mathbb{H}^n))}^{\,p},
\]
for some constant \(C>0\).

\noindent For \(\delta>1\), define
\[
S_\delta
=
\left\{
u\in L^\infty(0,T;L^2(\mathbb{H}^n)):
\|u\|_{L^\infty(0,T;L^2(\mathbb{H}^n))}
\leq
\delta\|u_0\|_{L^2(\mathbb{H}^n)}
\right\}.
\]

Clearly, \(S_\delta\) is a closed subset of \(L^\infty(0,T;L^2(\mathbb{H}^n))\). We shall show that, for a suitable choice of \(T=T^\ast\), the operator \(\mathcal{K}\) is a contraction on \(S_\delta\) and maps \(S_\delta\) into itself. The existence of a local solution then follows from the Banach fixed point theorem.

\noindent Let \(u\in S_\delta\). Then
\[
\|\mathcal{K}(u)\|_{L^\infty(0,T;L^2(\mathbb{H}^n))}
\leq
\|u_0\|_{L^2(\mathbb{H}^n)}
+
CT\delta^p\|u_0\|_{L^2(\mathbb{H}^n)}^p.
\]
Hence, whenever
\[
T
\leq
\frac{\delta-1}
{C\delta^p\|u_0\|_{L^2(\mathbb{H}^n)}^{p-1}},
\]
we obtain
\[
\|\mathcal{K}(u)\|_{L^\infty(0,T;L^2(\mathbb{H}^n))}
\leq
\delta\|u_0\|_{L^2(\mathbb{H}^n)}.
\]
Therefore, \(\mathcal{K}(S_\delta)\subseteq S_\delta\). It remains to show that \(\mathcal{K}\) is a contraction on \(S_\delta\).

\noindent Let \(u,v\in S_\delta\). Then
\begin{align*}
&\left\|
\int_0^\tau
e^{-(\tau-s)(\mathcal{L}+dI)}
\left(
|T_mu(s)|^p-|T_mv(s)|^p
\right)\,ds
\right\|_{L^2(\mathbb{H}^n)}
\\
&\leq
\int_0^\tau
\left\|
|T_mu(s)|^p-|T_mv(s)|^p
\right\|_{L^2(\mathbb{H}^n)}
\,ds
\\
&\leq
p\int_0^\tau
\left\|
|T_mu(s)|^{p-1}
+
|T_mv(s)|^{p-1}
\right\|_{L^{\frac{2p}{p-1}}(\mathbb{H}^n)}
\left\|
T_m(u(s)-v(s))
\right\|_{L^{2p}(\mathbb{H}^n)}
\,ds
\\
&\leq
p\int_0^\tau
\Big(
\|T_mu(s)\|_{L^{2p}(\mathbb{H}^n)}^{p-1}
+
\|T_mv(s)\|_{L^{2p}(\mathbb{H}^n)}^{p-1}
\Big)
\|T_m(u(s)-v(s))\|_{L^{2p}(\mathbb{H}^n)}
\,ds
\\
&\leq
pC^p
\int_0^\tau
\Big(
\|u(s)\|_{L^2(\mathbb{H}^n)}^{p-1}
+
\|v(s)\|_{L^2(\mathbb{H}^n)}^{p-1}
\Big)
\|u(s)-v(s)\|_{L^2(\mathbb{H}^n)}
\,ds.
\end{align*}

Hence,
\[
\|\mathcal{K}(u)-\mathcal{K}(v)\|_{L^\infty(0,T;L^2(\mathbb{H}^n))}
\leq
2pC^p\delta^{p-1}
T
\|u_0\|_{L^2(\mathbb{H}^n)}^{p-1}
\|u-v\|_{L^\infty(0,T;L^2(\mathbb{H}^n))}.
\]
Therefore, whenever
\[
T<T^\ast
=
\min
\left\{
\frac{\delta-1}
{C\delta^p\|u_0\|_{L^2(\mathbb{H}^n)}^{p-1}},
\,
\frac{1}
{2pC^p\delta^{p-1}
\|u_0\|_{L^2(\mathbb{H}^n)}^{p-1}}
\right\},
\]
the operator \(\mathcal{K}\) is a contraction on \(S_\delta\). Consequently, by the Banach fixed point theorem, \eqref{second heat equation} admits a local solution.
\end{proof}

\subsection{Nonlinear wave equations}

Next, we consider the following nonlinear wave equations in the space \(L^\infty(0,T;L^2(\mathbb{H}^n))\):
\begin{equation}\label{first wave equation}
u_{\tau\tau}=b(\tau)|T_mu|^p,\qquad 
u(0)=u_0,\quad u_\tau(0)=u_1,
\end{equation}
and
\begin{equation}\label{second wave equation}
u_{\tau\tau}+u_\tau=|T_mu|^p,\qquad 
u(0)=u_0,\quad u_\tau(0)=u_1,
\end{equation}
where \(T_m\) denotes the Fourier multiplier operator on \(L^2(\mathbb{H}^n)\) associated with a bounded function \(m\), \(1<p<\infty\), and \(b\) is a positive bounded function.

A function \(u\in L^\infty(0,T;L^2(\mathbb{H}^n))\) is called a global solution of \eqref{first wave equation} if
\begin{equation}\label{solution of first wave equation}
u=u_0+\tau u_1
+\int_0^\tau (\tau-s)b(s)|T_mu(s)|^p\,ds
\end{equation}
holds for every \(T<\infty\). Likewise, \(u\) is called a global solution of \eqref{second wave equation} if
\begin{equation}\label{solution of second wave equation}
u=u_0+(1-e^{-\tau})u_1
+\int_0^\tau (1-e^{-(\tau-s)})|T_mu(s)|^p\,ds
\end{equation}
holds for every \(T<\infty\).

In either case, \(u\) is called a local solution if the corresponding integral equation holds for some \(T^\ast>0\).

The following theorem establishes the well-posedness of the above nonlinear wave equations.
\begin{theorem}
Let \(1<p<\infty\), and let \(T_m\) be the Fourier multiplier operator associated with a function \(m\) satisfying
\[
\sup_{\alpha>0}
\alpha
\left(
\int_{\{a\in \Omega:\,|m(a)|>\alpha\}}
d\mu(a)
\right)^{\frac{1}{2}-\frac{1}{2p}}
<\infty.
\]
Then the following assertions hold:

\begin{enumerate}
\item The Cauchy problem \eqref{first wave equation} admits a local solution in \(L^\infty(0,T;L^2(\mathbb{H}^n))\) provided \(b\in L^2(0,T)\).

\item Assume that \(u_1=0\) and \(\gamma>3/2\). Suppose further that
\[
\|b\|_{L^2(0,T)}
\leq
cT^{-\gamma},
\qquad T>0,
\]
where \(c\) is a constant independent of \(T\). Then \eqref{first wave equation} admits a global solution in \(L^\infty(0,T;L^2(\mathbb{H}^n))\) for sufficiently small initial data \(u_0\).

\item The Cauchy problem \eqref{second wave equation} admits a local solution in \(L^\infty(0,T;L^2(\mathbb{H}^n))\).
\end{enumerate}
\end{theorem}
\begin{proof}
The proofs of parts \((1)\) and \((2)\) follow verbatim from \cite{Vishvesh-ruzhsnaky-k-a_multipliers}, and hence are omitted. We therefore restrict ourselves to proving the existence of a local solution for \eqref{second wave equation}.

\noindent The solution of \eqref{second wave equation} admits the form 
\begin{equation}\label{form of solution of second wave equation}
    \mathcal{K}(u):=u(\tau)=u_0+(1-e^{-\tau})u_1+\int_0^{\tau}(1-e^{-(\tau-s)})|T_mu(s)|^p\,ds.
\end{equation}
Now applying Minkowski's inequality in the above, we get 
\begin{align}
    \left\lVert \mathcal{K}(u) \right\rVert_{L^2\mathbb{H}^n} & \leq \left \lVert u_0 \right \rVert_{L^2(\mathbb{H}^n)}+ \left \lVert u_1 \right \rVert_{L^2(\mathbb{H}^n)}+ \left \lVert \int_0^\tau (1-e^{-(\tau-s)})\left(|T_mu(s)|^p \right) \right \rVert_{L^2(\mathbb{H}^n)}\,ds\nonumber\\ 
    &\leq \left \lVert u_0 \right \rVert_{L^2(\mathbb{H}^n)}+ \left \lVert u_1 \right \rVert_{L^2(\mathbb{H}^n)}+\int_0^\tau  \left \lVert |T_mu(s)|^p \right \rVert_{L^2(\mathbb{H}^n)}\,ds\nonumber\\
    &\leq \left \lVert u_0 \right \rVert_{L^2(\mathbb{H}^n)}+ \left \lVert u_1 \right \rVert_{L^2(\mathbb{H}^n)}+\int_0^\tau  \left \lVert T_mu(s) \right \rVert^p_{L^{2p}(\mathbb{H}^n)}\,ds\nonumber\\
    &\leq \left \lVert u_0 \right \rVert_{L^2(\mathbb{H}^n)}+ \left \lVert u_1 \right \rVert_{L^2(\mathbb{H}^n)}+C\int_0^\tau  \left \lVert u(s) \right \rVert^p_{L^{2}(\mathbb{H}^n)}\,ds,
\end{align}
where in the last inequality, we used that by the given hypothesis on \(m\), \(T_m: L^{2}(\mathbb{H}^n) \to L^{2p}(\mathbb{H}^n)\) is a bounded operator. Thus, we have 
we have 
   \begin{equation*}
       \lVert \mathcal{K}(u) \rVert_{L^{\infty}(0,T,L^2(\mathbb{H}^n))}\leq \lVert u_0 \rVert_{L^2(\mathbb{H}^n)}+ \lVert u_1 \rVert_{L^2(\mathbb{H}^n)} +CT\left\lVert u\right\rVert_{L^{\infty}(0,T,L^2(\mathbb{H}^n))}^p,
   \end{equation*}
 \noindent   For \(\delta>1\), consider the set 
\[S_\delta=\left\{ u \in L^{\infty}(0,T,L^2(\mathbb{H}^n)):\left\lVert u\right\rVert_{L^{\infty}(0,T,L^2(\mathbb{H}^n))}\leq \delta  \left(\lVert u_0 \rVert_{L^2(\mathbb{H}^n)}+\lVert u_1\rVert_{L^2(\mathbb{H}^n)} \right)\right\}.\]
\noindent Next, continuing as in the case of Theorem \ref{theorem for heat equations}, we can obtain the desired result.

\end{proof}

\section{\(L^p-L^p\) boundedness}\label{sec5}

This section is devoted to the study of the \(L^p\)-\(L^p\) boundedness of Fourier multiplier operators in the range \(1<p<\infty\).

Recall that, on \(\mathbb{R}^n\), the Fourier multiplier operator associated with a bounded function \(\sigma\) is defined by 
\begin{equation*}
    \widehat{T_{\sigma}f}(\xi)=\sigma(\xi)\widehat{f}(\xi),\;\; \xi \in \mathbb{R}^n.
\end{equation*}
By the Plancherel theorem, it is immediate that \(T_\sigma\) is bounded on \(L^2(\mathbb{R}^n)\). H\"ormander \cite{hormander-estimates_of_translation} also established sufficient conditions for the \(L^p\)-boundedness of \(T_\sigma\). More precisely, he proved that if \(k\geq n/2+1\), and \(\sigma\in C^k(\mathbb{R}^n\setminus\{0\})\) satisfies
\[
\sup_{R>0}
R^{|\beta|-\frac{n}{2}}
\left(
\int_{R<|\xi|<2R}
|D^\beta \sigma(\xi)|^2\,d\xi
\right)^{1/2}
\leq C
\]
for all multi-indices \(\beta\) with \(|\beta|\leq k\), then \(T_\sigma\) extends to a bounded operator on \(L^p(\mathbb{R}^n)\), \(1<p<\infty\).

We now consider the operator \(T_m\) on \(L^p(\mathbb{H}^n)\), as defined in \eqref{operator T_m}. We begin by showing that \(T_m\) can be realized as a convolution operator. 
\begin{align}
T_mf(z,t)
&=\mathcal{F}^{-1}(m\mathcal{F}(f))(z,t)\nonumber\\
&=
\int_{\Omega}
\int_{\mathbb{C}^n}
c_{n,k}\,
m(\lambda,(2k+n)|\lambda|)
\mathcal{F}(f)(\lambda,(2k+n)|\lambda|,w)
e_a\big((-w,0)(z,t)\big)
\,dw\,d\nu_2(a)
\nonumber\\
&=
\int_{\mathbb{H}^n}
f(u,s)
\Bigg(
\int_{\Omega}
\int_{\mathbb{C}^n}
c_{n,k}^2
m(\lambda,(2k+n)|\lambda|)
\varphi^{n-1}_{k,\lambda}(-u+w)
\varphi^{n-1}_{k,\lambda}(-w+z)
\nonumber\\
&\hspace{4cm}\times
e^{i\lambda\left(t-s-\frac12\operatorname{Im}(w\bar z+u\bar w)\right)}
\,dw\,d\nu_2(a)
\Bigg)
\,du\,ds
\nonumber\\
&=
(2\pi)^n
\int_{\mathbb{H}^n}
f(u,s)
\Bigg(
\int_{\Omega}
c_{n,k}^2
|\lambda|^{-n}
m(\lambda,(2k+n)|\lambda|)
\varphi^{n-1}_{k,\lambda}(z-u)
\nonumber\\
&\hspace{4cm}\times
e^{i\lambda\left(t-s-\frac12\operatorname{Im}(u\bar z)\right)}
\,d\nu_2(a)
\Bigg)
\,du\,ds.
\end{align}
If we define the kernel \(K\) by
\begin{align}
K(z,t)
&=
(2\pi)^n
\int_{\Omega}
c_{n,k}^2
|\lambda|^{-n}
m(\lambda,(2k+n)|\lambda|)
\varphi^{n-1}_{k,\lambda}(z)
e^{i\lambda t}
\,d\nu_2(a)
\nonumber\\
&=
(2\pi)^{-n-1}
\int_{-\infty}^{\infty}
\sum_{k=0}^{\infty}
|\lambda|^n
m(\lambda,(2k+n)|\lambda|)
\varphi^{n-1}_{k,\lambda}(z)
e^{i\lambda t}
\,d\lambda,
\label{expression of the kernel}
\end{align}
then \(T_m\) admits the convolution representation
\[
T_mf=f\ast K.
\]

Having obtained an explicit expression for the convolution kernel \(K(z,t)\), we now relate the operator \(T_m\) to the framework developed in \cite{Muller-Marcinkiewicz_multipliers-I}.

\noindent Following the terminology of \cite{Muller-Marcinkiewicz_multipliers-I}, for \(k\geq 1\), define
\[
\Delta m(\lambda,(2k+n)|\lambda|)
=
m(\lambda,(2k+n)|\lambda|)
-
m(\lambda,(2(k-1)+n)|\lambda|).
\]

We now apply \cite[Theorem 2.2]{Muller-Marcinkiewicz_multipliers-I}. The authors of \cite{Muller-Marcinkiewicz_multipliers-I} studied joint multiplier operators on product groups. In the case of the Heisenberg group, they considered multipliers associated with the sublaplacian \(\mathcal{L}\) and the central vector field \(T\), and established the \(L^p\)-boundedness of the corresponding convolution operators.

Observe that the convolution kernel \(K(z,t)\) obtained in \eqref{expression of the kernel} coincides with the kernel associated with the operator considered in \cite[Theorem 2.2]{Muller-Marcinkiewicz_multipliers-I}. Consequently, we obtain the following theorem.
\begin{theorem}
Let \(m\) be a bounded function on \(\Omega\). Suppose that
\[
\left|
(k\Delta)^{\alpha}
\left(
\lambda\frac{\partial}{\partial\lambda}
\right)^{\beta}
m(\lambda,(2k+n)|\lambda|)
\right|
\leq C,
\]
for all integers \(\alpha,\beta\leq N\), where \(N\) is sufficiently large. Then the operator \(T_m\) extends to a bounded operator on \(L^p(\mathbb{H}^n)\) for every \(1<p<\infty\).
\end{theorem}
\section*{Acknowledgement}
    The second author gratefully acknowledge the Indian Institute of Technology Delhi for providing the Institute Assistantship.

\section*{Conflict of Interest}
The authors declare that they have no competing interests.

\bibliographystyle{acm}
\bibliography{ref_multipliers}
\end{document}